\documentclass{daj}

\usepackage{amsmath,amsthm,amssymb,booktabs,fullpage}
\usepackage{txfonts,eucal,exscale}
\usepackage{caption}
\newtheorem{thm}{Theorem}
\newtheorem{lemma}{Lemma}[section]
\newtheorem{pro}{Proposition}[section]
\newtheorem{cor}{Corollary}[section]

\theoremstyle{definition}
\newtheorem{remark}{Remark}[section]

\newcommand{\eps}{\varepsilonup}
\renewcommand{\alpha}{\alphaup}
\renewcommand{\beta}{\betaup}
\renewcommand{\gamma}{\gammaup}
\renewcommand{\lambda}{\lambdaup}
\renewcommand{\mu}{\muup}
\renewcommand{\xi}{\xiup}
\renewcommand{\psi}{\psiup}
\renewcommand{\phi}{\phiup}
\renewcommand{\eta}{\etaup}
\renewcommand{\theta}{\thetaup}

\numberwithin{equation}{section}

\dajAUTHORdetails{%
  title = {Improved $\ell^p$-Boundedness for Integral $k$-Spherical Maximal Functions},
  author = {Theresa C. Anderson, Brian Cook, Kevin Hughes, and Angel Kumchev},
  plaintextauthor = {Theresa C. Anderson, Brian Cook, Kevin Hughes, and Angel Kumchev},
  plaintexttitle = {Improved lp-Boundedness for Integral k-Spherical Maximal Functions}, 
  runningtitle = {Improved Bounds for Integral Maximal Functions}, 
  runningauthor = {T. C. Anderson, B. Cook, K. Hughes, and A. Kumchev},
  keywords = {Maximal functions, integral averages, surface measures, Fourier transforms, circle method, exponential sums.}
  }
   
\dajEDITORdetails{%
   year={2018},
   number={10},
   received={2 August 2017},   
   published={29 May 2018},  
   doi={10.19086/da.3675},       
}   

\begin{document}

\begin{frontmatter}[classification=text]


\author[TA]{Theresa Anderson}
\author[BC]{Brian Cook}
\author[KH]{Kevin Hughes}
\author[AK]{Angel Kumchev}

\begin{abstract}
We improve the range of $\ell^p(\mathbb Z^d)$-boundedness of the integral $k$-spherical maximal functions introduced by Magyar. The previously best known bounds for the full $k$-spherical maximal function require the dimension $d$ to grow at least cubically with the degree $k$. Combining ideas from our prior work with recent advances in the theory of Weyl sums by Bourgain, Demeter, and Guth and by Wooley, we reduce this cubic bound to a quadratic one. As an application, we improve upon bounds in the authors' previous work \cite{ACHK:WaringGoldbach} on the ergodic Waring--Goldbach problem, which is the analogous problem of $\ell^p(\mathbb Z^d)$-boundedness of the $k$-spherical maximal functions whose coordinates are restricted to prime values rather than integer values.
\end{abstract}
\end{frontmatter}

\section{Introduction}

Our interest lies in proving $\ell^p(\mathbb Z^d)$-bounds for the integral $k$-spherical maximal functions when $k \geq 3$. These maximal functions are defined in terms of their associated averages, which we now describe. Define a positive definite function $\mathfrak f$ on $\mathbb R^d$ by 
\[ \mathfrak f(\mathbf x) = \mathfrak f_{d,k}(\mathbf x) := |x_1|^k + \dots + |x_d|^k, \]
and note that when $\mathbf x \in \mathbb R_+^d$, $\mathfrak f(\mathbf x)$ is the diagonal form $x_1^k + \dots + x_d^k$. For $\lambda \in \mathbb N$, let $R(\lambda)$ denote the number of integral solutions to the equation
\begin{equation}\label{eq:numberofintegralsolutions}
\mathfrak f_{d,k}(\mathbf x) = \lambda. 
\end{equation}
When $R(\lambda)>0$, define the normalized arithmetic surface measure 
\[ \sigma_\lambda({\mathbf x}) := \frac{1}{R(\lambda)}{\mathbf 1}_{\{{\mathbf y} \in \mathbb Z^d: \mathfrak f(\mathbf y)=\lambda\}}({\mathbf x}). \] 
We are interested in averages given by convolution with these measures:
\[ \sigma_\lambda * f(\mathbf x) = \frac 1{R(\lambda)}\sum_{\mathfrak f (\mathbf y) = \lambda} f(\mathbf x - \mathbf y) \]
for functions $f: \mathbb Z^d \to \mathbb C$.

We know from the literature on Waring's problem that as $\lambda \to \infty$, one has the asymptotic 
\begin{equation}\label{eq:asymptoticnumberofsolutions}
R(\lambda) \sim \mathfrak S_{d,k}(\lambda) \lambda^{{d/k-1}},
\end{equation} 
where $\mathfrak S_{d,k}(\lambda)$ is a convergent product of local densities:
\[  \mathfrak S_{d,k}(\lambda) = \prod_{p \le \infty} \mu_p(\lambda).  \]
Here $\mu_p(\lambda)$ with $p < \infty$ is related to the solubility of \eqref{eq:numberofintegralsolutions} over the $p$-adic field $\mathbb Q_p$, and $\mu_\infty(\lambda)$ to solubility over the reals. It is known that when $d$ is sufficiently large in terms of~$k$, one has 
\begin{equation}\label{eq:singularseries} 
1 \lesssim \mathfrak S_{d,k}(\lambda) \lesssim 1.
\end{equation} 
In particular, these bounds hold for $d \ge 4k$ when $k \ge 4$ is a power of $2$, and for $d \ge \frac 32k$ otherwise (see Theorems~4.3 and 4.6 in Vaughan \cite{Vaughan}). 

Throughout the paper we use the notation $f(x) \lesssim g(x)$ or $g(x) \gtrsim f(x)$ to mean that there exists a constant $C > 0$ so that $|f(x)| \leq C|g(x)|$ for all sufficiently large $x \geq 0$. 
The implicit constant $C$ may depend on `inessential' or fixed parameters, but will be independent of `$x$';  below the implicit constants will often depend on the parameters $k,d$ and $p$. 
For instance, \eqref{eq:singularseries} means that there exists positive constants $C_1$ and $C_2$ depending on $d$ and $k$ so that $C_1 \leq \mathfrak S_{d,k}(\lambda) \leq C_2$. 

In view of \eqref{eq:asymptoticnumberofsolutions} and \eqref{eq:singularseries}, we may replace the convolution $\sigma_\lambda * f$ above by the average
\begin{equation}\label{integralavg} 
A_\lambda f(\mathbf x) := \lambda^{1-d/k} \sum_{\mathfrak f (\mathbf y) = \lambda} f(\mathbf x - \mathbf y).
\end{equation}
Our $k$-spherical maximal function is then defined, for $\mathbf x \in \mathbb Z^d$, as the pointwise supremum of all averages
\begin{equation}\label{def_A*}
A_* f(\mathbf x) := \sup_{\lambda \in \mathbb N} |A_\lambda f(\mathbf x)|.
\end{equation}
Variants of this maximal function were introduced by Magyar \cite{Magyar_dyadic} and studied later in~\cite{MSW, Magyar:ergodic, Ionescu, AS, Hughes_Vinogradov, Hughes_restricted}. In particular,  Magyar, Stein and Wainger \cite{MSW} considered the above maximal function when $k=2$ and $d \geq 5$ and proved that it is bounded on $\ell^p(\mathbb Z^d)$ when $p > \frac{d}{d-2}$. This result is sharp except at the endpoint, for which the restricted weak-type bound was proved later by Ionescu~\cite{Ionescu}. To the best of our knowledge, the sharpest results on the boundedness of $A_*$ for degrees $k \geq 3$ are those obtained by the third author in \cite{Hughes_restricted}. In the present paper, we give a further improvement.

For $k \ge 3$, define 
\begin{equation}\label{eq:def_d0}
d_0(k) := k^2 - \max_{2 \leq j \leq k-1} \left\{ \frac {k j - \min(2^j+2,j^2+j)}{k-j+1} \right\}.
\end{equation} 
Further, set $\tau_k = \max\left\{ 2^{1-k},  (k^2 - k)^{-1} \right\}$ and define the function $\delta_0(d,k)$ by 
\[ k \, \delta_0(d,k) := \begin{cases} (d-d_0)/(k^2+k - d_0) & \text{if } d_0(k) \le d \le k^2+k, \\ 1 + (d-k^2-k)\tau_k &\text{if } d > k^2+k. \end{cases} \]
Finally, define
\[ p_0(d,k):=\max \left\{ \frac{d}{d-k}, 1+\frac{1}{1+2\delta_0(d,k)} \right\}. \]
We remark that when $d > d_0(k)$, $p_0(d,k)$ always lies in the range $(1,2)$. Our main result is the following.

\begin{thm}\label{thm:newmax}
If $k \geq 3$, $d > d_0(k)$ and $p>p_0(d,k)$, then the maximal operator $A_*$, defined by \eqref{def_A*}, is bounded on \(\ell^p(\mathbb Z^d)\): that is,
\[ \| A_*f \|_{\ell^p(\mathbb Z^d)} \lesssim \| f \|_{\ell^p(\mathbb Z^d)} . \]
\end{thm}

Let $d_0^*(k) := 1 + \lfloor d_0(k) \rfloor$ denote the least dimension in which Theorem \ref{thm:newmax} establishes that $A_*$ is bounded on $\ell^2(\mathbb Z^d)$. We emphasize that for large $k$, we have $d_0^*(k) = k^2 - k + O(k^{1/2})$, whereas in previous results, such as the work of the third author \cite{Hughes_Vinogradov,Hughes_restricted}, one required $d > k^3-k^2$. 
While our results improve on these previous results by a factor of the degree $k$, the conjecture is that the maximal function is bounded on $\ell^2(\mathbb Z^d)$ for $d \gtrsim k$ (see \cite{Hughes_Vinogradov} for details), but such a result appears to be way beyond the reach of present methods. 

It is also instructive to compare $d_0^*(k)$ to the known bounds for the function $\tilde G(k)$ in the theory of Waring's problem (defined as the least value of $d$ for which the asymptotic formula \eqref{eq:asymptoticnumberofsolutions} holds). It transpires that the values of $d_0^*(k)$ match the best known upper bounds on $\tilde G(k)$ for all but a handful of small values of $k$, and even in those cases, we miss the best known bound on $\tilde G(k)$ only by a dimension or two. For an easier comparison, we list the numerical values of $d_0^*(k)$, $k \le 10$, their respective analogues in earlier work, and the bounds on $\tilde G(k)$ in Table \ref{table1}. 
\begin{table}[!htb]
\begin{center}
\begin{tabular}{ccccccccc}
\toprule
$k$                &  3 &  4 &  5 &   6 &   7 &   8 &  9  &  10 \\
\midrule
$d_0^*(k)$         & 10 & 16 & 24 &  35 &  47 &  62 &  79 &  97 \\
$d_1^*(k)$         & 13 & 33 & 81 & 181 & 295 & 449 & 649 & 901 \\
$\tilde G(k) \le $ &  8 & 15 & 23 &  34 &  47 &  61 &  78 &  97 \\
\bottomrule
\end{tabular}
\end{center}
\bigskip
\caption{Comparison between the values of $d_0^*(k)$ in Theorem \ref{thm:newmax}, the corresponding values $d_1^*(k)$ in the work of Hughes \cite{Hughes_Vinogradov}, and the known upper bounds on $\tilde G(k)$ in Bourgain \cite{Bourgain_Vinogradov} and Vaughan \cite{Vaughan_cubes}.}
\label{table1}
\end{table}

A key ingredient in the proof of Theorem \ref{thm:newmax} and its predecessors is an approximation formula generalizing \eqref{eq:asymptoticnumberofsolutions}. First introduced in \cite{MSW} when $k=2$, such approximations are obtained for the average's corresponding Fourier multiplier:
\begin{equation}\label{eq1.3}
\widehat{A_\lambda}({\boldsymbol\xiup}) = \lambda^{1-d/k} \sum_{\mathfrak f(\mathbf x)=\lambda}  e(\mathbf x \cdot \boldsymbol\xi),
\end{equation}
where $\boldsymbol\xi \in \mathbb T^d$ and $e(z) = e^{2\pi iz}$. 

We need to introduce some notation in order to state our approximation formula. Given an integer $q \ge 1$, we write $\mathbb Z_q = \mathbb Z/q\mathbb Z$ and $\mathbb Z_q^*$ for the group of units; we also write $e_q(x) = e(x/q)$. The $d$-dimensional Gauss sum of degree $k$ is defined as 
\[ G(q; a, \mathbf b) = q^{-d}\sum_{\mathbf x \in \mathbb Z_q^d} e_q ( a\mathfrak f(\mathbf x) + \mathbf b \cdot \mathbf x) \]
for $a \in \mathbb Z_q$ and $\mathbf b \in \mathbb Z_q^d$. If $\Sigma_\lambda$ denotes the surface in $\mathbb R_+^d$ defined by \eqref{eq:numberofintegralsolutions} and $dS_\lambda(\mathbf x)$ denotes the induced Lebesgue measure on $\Sigma_\lambda$, we define a continuous surface measure $d\sigma_\lambda(\mathbf x)$ on $\Sigma_\lambda$ by  
\[ 
d\sigma_{\lambda}(\mathbf x) 
:= \lambda^{1-d/k} \frac {dS_{\lambda}(\mathbf x)} {|\nabla\mathfrak f(\mathbf x)|}.
\]
We note that $d\sigma_\lambda$ is essentially a probability measure on $\Sigma_\lambda$ for all $\lambda$. We also fix a smooth bump function $\psi$, which is $1$ on \(\left[-\frac 18,\frac 18\right]^d\) and supported in \(\left[-\frac 14, \frac 14\right]^d\). Finally, we write $\widetilde{\mu}$ for the $\mathbb R^d$-Fourier transform of a measure $\mu$ on $\mathbb R^d$ and $\widehat{f}$ for the $\mathbb Z^d$-Fourier transform (which has domain $\mathbb T^d$) of a function $f$ on $\mathbb Z^d$. 

\begin{thm}[Approximation Formula] \label{thm:approximation_formula}
If $k \geq 3$, $d > d_0(k)$ and $\lambda \in \mathbb N$ is sufficiently large, then one has
\[ \widehat{A_\lambda}(\boldsymbol\xi) = \sum_{q=1}^{\infty} \sum_{a \in \mathbb Z_q^*} e_q( -\lambda a) \sum_{\boldsymbol w \in \{ \pm 1 \}^d} \sum_{\mathbf b \in \mathbb Z^d} G(q; a, \boldsymbol w\mathbf b) \psi(q\boldsymbol\xi-\mathbf b) \widetilde{d\sigma_{\lambda}}(\boldsymbol w(\boldsymbol\xi-q^{-1}\mathbf b)) + \widehat {E_\lambda}(\boldsymbol\xi), \]
where the error terms $\widehat{E_\lambda}$ are the multipliers of convolution operators satisfying the dyadic maximal inequality
\begin{equation}\label{eq:dyadic_error_bound}
\left\| \sup_{\lambda \in [\Lambda/2, \Lambda)} |E_\lambda| \right\|_{\ell^2(\mathbb Z^d) \to \ell^2(\mathbb Z^d)} \lesssim \Lambda^{-\delta}
\end{equation}
for each $\Lambda \ge 1$ and all sufficiently small $\delta > 0$.
\end{thm}

Our Approximation Formula takes the same shape as those in \cite{Hughes_Vinogradov} and \cite{Hughes_restricted}, but with an improved error term that relies on two recent developments: 
\begin{itemize}
\item
the underlying analytic methods were improved in the authors' previous work \cite{ACHK:WaringGoldbach},
\item
and the recent resolution of the main conjecture about Vinogradov's mean value integral \cite{Wooley_cubic, BourgainDemeterGuth} and related refinements of classical mean value estimates \cite{Wooley_asymptotic, Bourgain_Vinogradov}.
\end{itemize}
The most dramatic improvement follows from our improved analytic methods originating in \cite{ACHK:WaringGoldbach} where we improve the range of $\ell^2(\mathbb{Z}^d)$ by a factor of the degree $k$. 
In \cite{Hughes_restricted}, this sort of improvement - which also used the recent resolution of the Vinogradov mean values theorems \cite{Wooley_asymptotic, Bourgain_Vinogradov} - was limited to maximal functions over sufficiently sparse sequences. 
Here, our bounds supersede those for integral $k$-spherical maximal functions over sparse sequences in \cite{Hughes_restricted} because our treatment of the minor arcs in the error term is more efficient. 
The reader may compare Lemmas~\ref{lemma:meanvaluereduction} and \ref{lemma:L2bound} below to Lemmas~2.1 and 2.2 of \cite{Hughes_restricted} to determine the efficacy of our method here. 
Consequently, \cite{Wooley_asymptotic} and \cite{Bourgain_Vinogradov} allow us to further improve slightly upon a more direct application of the Vinogradov mean value theorems from \cite{Wooley_cubic} and \cite{BourgainDemeterGuth}. 
One minor drawback is that in our method the $\epsilon$-losses in \cite{Wooley_cubic,BourgainDemeterGuth} do not allow us to deduce endpoint bounds. 

As an application, we deduce that the maximal function of the "ergodic Waring--Goldbach problem" introduced in our recent work \cite{ACHK:WaringGoldbach} is bounded on the same range of spaces as above. That maximal function is associated with averages where, instead of sampling over integer points, we sample over points where all coordinates are prime. To be precise, let $R^*(\lambda)$ denote the number of prime solutions to the equation \eqref{eq:numberofintegralsolutions} weighted by logarithmic factors: that is, 
\[ R^*(\lambda) := \sum_{ \mathfrak f(\mathbf x) = \lambda} {\mathbf 1}_{\mathbb P^d}(\mathbf x)(\log x_1) \cdots (\log x_d), \]
where $\mathbf 1_{\mathbb P^d}$ is the indicator function of vectors $\mathbf x \in \mathbb{Z}^d$ with all coordinates prime. When $R^*(\lambda)>0$, define the normalized arithmetic surface measure 
\[ \omega_\lambda({\mathbf x}) := \frac{1}{R^*(\lambda)}{\bf 1}_{\{{\mathbf y} \in \mathbb P^d: \mathfrak f(\mathbf y)=\lambda\}}({\mathbf x}) (\log x_1) \cdots (\log x_d) \] 
and the respective convolution operators  
\begin{equation}\label{primeavg}
W_\lambda f := \omega_\lambda * f.
\end{equation}
Similarly to \eqref{eq:asymptoticnumberofsolutions}, we know that as $\lambda \to \infty$, one has the asymptotic 
\[ R^*(\lambda) = \big(\mathfrak S_{d,k}^*(\lambda) + o(1)\big) \lambda^{{d/k-1}}, \]
where $\mathfrak S_{d,k}^*(\lambda)$ is a product of local densities similar to $\mathfrak S_{d,k}(\lambda)$ above. Moreover, when $d > 3k$ and $\lambda$ is restricted to a particular arithmetic progression $\Gamma_{d,k}$, we have $1 \lesssim \mathfrak S_{d,k}^*(\lambda) \lesssim 1$, and the above estimate turns into a true asymptotic formula for $R^*(\lambda)$ (see the introduction and references in \cite{ACHK:WaringGoldbach}). By Theorem~6 of~\cite{ACHK:WaringGoldbach} and Theorem~\ref{thm:newmax} above, we immediately obtain the following result. Here, as in \cite{ACHK:WaringGoldbach}, $d_1(3) = 13$ and $d_1(k) = k^2+k+3$.

\begin{thm}
If $k \geq 3$, $d \ge d_1(k)$ and $p > p_0(d,k)$, then the maximal function defined by
\[ W_* f(\mathbf x) := \sup_{\lambda \in \Gamma_{d,k}} |W_\lambda f(\mathbf x)|, \]
is bounded on $\ell^p(\mathbb Z^d)$.
\end{thm}
\noindent 
As another application one may give analogous improvements of the ergodic theorems obtained in~\cite{Hughes_Vinogradov}, but we do not consider this here.

To establish our theorems, we follow the paradigm in \cite{MSW} and strengthen the connection to Waring's problem as initiated in \cite{Hughes_restricted} by using a lemma from \cite{ACHK:WaringGoldbach}; we then use recent work on Waring's problem to obtain improved bounds. 
We remark that \cite{Stein_Wainger_discrete_fractional_integration_revisited} and \cite{Pierce_mean_values} previously connected mean values (Hypothesis $K^*$ and Vinogradov's mean value theorems respectively) to discrete fractional integration. 
In Section~\ref{section:previous}, we outline the proofs of Theorems \ref{thm:newmax} and \ref{thm:approximation_formula}; we recall some results from \cite{Hughes_Vinogradov,Magyar_dyadic} and state the key propositions required in the proofs. The remaining sections establish the propositions. 
In Section \ref{section:minorarcs} we deal with the minor arcs; particularly, in Section~\ref{section:VMVT}, we use the recent work of Bourgain, Demeter and Guth \cite{BourgainDemeterGuth} on Vinogradov's mean value theorem and a method of Wooley \cite{Wooley_asymptotic} for estimation of mean values over minor arcs. 
In Section~\ref{section:majorarcs}, we establish the relevant major arc approximations. Finally, in Section \ref{sec:5}, we establish the boundedness of the maximal function associated with the main term in the Approximation Formula.

\section{Outline of the Proof}\label{section:previous}

Since $A_*$ is trivially bounded on $\ell^{\infty}(\mathbb Z^d)$, we may assume through the rest of the paper that $p \le 2$. 
Fix $\Lambda \in \mathbb N$ and consider the dyadic maximal operator
\[ A_{\Lambda} f := \sup_{\lambda \in [\Lambda/2,\Lambda)} |A_{\lambda}f|. \]
When $\lambda \le \Lambda$, we have
\[ A_\lambda f = \lambda^{1-d/k} \int_{\mathbb T} (h_{\Lambda}(\theta) * f)e(-\lambda\theta) \, d\theta, \]
where 
\[ h_{\Lambda}(\theta) = h_{\Lambda}(\theta; \mathbf x) := e(\theta \mathfrak f(\mathbf x))\mathbf 1_{[-N,N]^d}(\mathbf x) \quad \text{with} \quad N = \Lambda^{1/k}. \] 
This representation allows us to decompose $A_\lambda$ into operators of the form
\begin{equation}\label{eq:integral} 
A_\lambda^{\mathfrak B} f = \lambda^{1-d/k} \int_{\mathfrak B} (h_{\Lambda}(\theta) * f)e(-\lambda\theta) \, d\theta, 
\end{equation}
for various measurable sets $\mathfrak B \subseteq \mathbb T$.

Our decomposition of $A_{\lambda}$, is inspired by the Hardy--Littlewood circle method. When $q \in \mathbb N$ and $0 \le a \le q$, we define the major arc ${\mathfrak M}(a/q)$ by
\[ {\mathfrak M}(a/q) = \left[ \frac aq - \frac 1{4kqN^{k-1}},\frac aq + \frac 1{4kqN^{k-1}} \right]. \]
We then decompose $\mathbb T$ into sets of \emph{major} and \emph{minor arcs}, given by
\[ {\mathfrak M} = {\mathfrak M}(\Lambda) = \bigcup_{q \leq N} \bigcup_{a \in \mathbb Z_q^*} {\mathfrak M}(a/q) \quad \text{ and } \quad {\mathfrak m}(\Lambda) = \mathbb T \setminus {\mathfrak M}(\Lambda). \]
Since the major arcs ${\mathfrak M}(a/q)$ are disjoint, this yields a respective decomposition of $A_\lambda$ as 
\[ A_\lambda = \sum_{q \le N} \sum_{a \in \mathbb Z_q^*} A_\lambda^{a/q} + A_\lambda^{{\mathfrak m}}, \]
where $A_\lambda^{a/q} := A_\lambda^{{\mathfrak M}(a/q)}$. 

We will use the notations $A_*^{\mathfrak B}$ and $A_{\Lambda}^{\mathfrak B}$ to denote the respective maximal functions obtained from the operators $A_\lambda^{\mathfrak{B}}$. For example, from \eqref{eq:integral} and the trivial bound for the trigonometric polynomial $h_{\Lambda}(\theta) * f$, we obtain the trivial $\ell^1$-bound
\begin{equation}\label{eq:l1}
\left\| A_{\Lambda}^{\mathfrak B} \right\|_{\ell^1(\mathbb Z^d) \to \ell^1(\mathbb Z^d)} \lesssim \Lambda |\mathfrak B|.
\end{equation}
In Section~\ref{section:minorarcs}, we analyze the minor arc term and prove the following result.

\begin{pro}\label{lemma:minor_arc_approximation}
If $k \geq 3$, $d > d_0(k)$, and $\Lambda \gtrsim 1$, then
\begin{equation}\label{eq:minor_arc_approximation}
\left\| A_{\Lambda}^{{\mathfrak m}} \right\|_{\ell^2(\mathbb Z^d) \to \ell^2(\mathbb Z^d)} \lesssim \Lambda^{-\delta}
\end{equation}
for all $\delta \in (0,\delta_0(d,k))$.
\end{pro}

For $1 < p \le 2$, interpolation between \eqref{eq:l1} and \eqref{eq:minor_arc_approximation} yields 
\[ \left\| A_{\Lambda}^{\mathfrak m} \right\|_{\ell^p(\mathbb Z^d) \to \ell^p(\mathbb Z^d)} \lesssim \Lambda^{-\alpha_p} \]
with $\alpha_p = 2(1+\delta)(1-1/p) - 1$. When $p_0(d,k) < p \le 2$, we have $\alpha_p > 0$, and hence,
\begin{equation}\label{eq:main_minor_arc}
\left\| A_*^{\mathfrak m} \right\|_{\ell^p(\mathbb Z^d) \to \ell^p(\mathbb Z^d)} \le \sum_{\Lambda = 2^j \gtrsim 1} \left\| A_{\Lambda}^{\mathfrak m} \right\|_{\ell^p(\mathbb Z^d) \to \ell^p(\mathbb Z^d)} \lesssim \sum_{\Lambda = 2^j \gtrsim 1} \Lambda^{-\alpha_p} \lesssim 1.
\end{equation}  

The estimation of the major arc terms is more challenging, because an analogue of~\eqref{eq:minor_arc_approximation} does not hold for $A_{\Lambda}^{{\mathfrak M}}$. Still, it is possible to establish a slightly weaker version of Proposition~\ref{lemma:minor_arc_approximation}. The following result was first established by Magyar~\cite{Magyar_dyadic}, for $d \ge 2^k$, and then extended by the third author \cite{Hughes_Vinogradov} in the present form.

\begin{pro}\label{lemma:major_arc_bound}
If $k \geq 3$, $d > 2k$, $\frac{d}{d-k} < p \le 2$, and $\Lambda \gtrsim 1$, one has 
\begin{equation}\label{eq:major_arc_bound}
\left\| A_{\Lambda}^{{\mathfrak M}} \right\|_{\ell^p(\mathbb Z^d) \to \ell^p(\mathbb Z^d)} \lesssim 1.
\end{equation}
\end{pro}

This proposition suffices to establish the $\ell^p$-boundedness of the dyadic maximal functions $A_{\Lambda}$ (this is the main result of Magyar \cite{Magyar_dyadic}), but falls just short of what is needed for an equally quick proof of Theorem~\ref{thm:newmax}. Instead, we use Theorem~\ref{thm:approximation_formula} to approximate $A_*^{{\mathfrak M}}$ by a bounded operator. Let $M_\lambda^{a/q}$ denote the convolution operator on $\ell^p(\mathbb Z^d)$ with Fourier multiplier
\[ \widehat{M_\lambda^{a/q}}(\boldsymbol\xi) :=  e_q( -\lambda a) \sum_{\boldsymbol w \in \{ \pm 1 \}^d} \sum_{\mathbf b \in \mathbb Z^d} G(q; a, \boldsymbol w\mathbf b) \psi(q\boldsymbol\xi-\mathbf b) \widetilde{d\sigma_{{\lambda}}}(\boldsymbol w(\boldsymbol\xi-q^{-1}\mathbf b)), \]
and define
\[ M_\lambda := \sum_{q \in \mathbb N} \sum_{a \in \mathbb Z_q^*} M_\lambda^{a/q} \quad \text{ and } \quad M_* := \sup_{\lambda \in \mathbb N} |M_\lambda|. \] 
In Section~\ref{section:majorarcs}, we will handle the major arc approximations and prove the following proposition. 

\begin{pro}\label{lemma:majorarcs}
If $k \ge 3$, $d > 2k$, and $\frac{d}{d-k} < p \le 2$, then there exists an exponent $\beta_p = \beta_p(d,k) > 0$ such that 
\begin{equation}\label{eq:majorarcs}
\left\| \sup_{\lambda \in [\Lambda/2, \Lambda)} \big| A_\lambda^{{\mathfrak M}}-M_\lambda \big| \right\|_{\ell^p(\mathbb Z^d) \to \ell^p(\mathbb Z^d)} \lesssim \Lambda^{-\beta_p}.
\end{equation}
\end{pro}

\begin{remark}
Theorem \ref{thm:approximation_formula} is an immediate consequence of Propositions \ref{lemma:minor_arc_approximation} and~\ref{lemma:majorarcs}. Moreover, since when $p = 2$ and $d \ge \frac 52k$, inequality \eqref{eq:majorarcs} holds for any $\beta_2 < 1/(2k)$ (see inequality \eqref{eq:4.8} and the comments after it), the error bound \eqref{eq:dyadic_error_bound} holds for all $\delta$ with $0 < \delta < \min \{ \delta_0(d,k), 1/(2k) \}$. 
\end{remark}  

When we sum \eqref{eq:majorarcs} over dyadic $\Lambda = 2^j$, we deduce that
\[ \left\| \sup_{\lambda \gtrsim 1} \big| A_\lambda^{{\mathfrak M}}-M_\lambda \big| \right\|_{\ell^p(\mathbb Z^d) \to \ell^p(\mathbb Z^d)} \lesssim 1. \]
Combining this bound and \eqref{eq:main_minor_arc}, we conclude that the $\ell^p$-boundedness of the maximal operator $A_*$ follows from the $\ell^p$-boundedness of $M_*$. The following proposition, which we establish in Section~\ref{sec:5}, then completes the proof of Theorem \ref{thm:newmax}.

\begin{pro}\label{lemma:mainterm}
If $k \ge 3$, $d> 2k$, and $\frac{d}{d-k} < p \le 2$, then $M_*$ is bounded on $\ell^p(\mathbb Z^d)$.  
\end{pro}

\section{Minor Arc Analysis}
\label{section:minorarcs}

Our minor arc analysis splits naturally in two steps. The first step is a reduction to mean value estimates related to Waring's problem; for this we use a technique introduced in \cite{ACHK:WaringGoldbach}. We then apply recent work on Waring's problem to estimate the relevant mean values and to prove Proposition~\ref{lemma:minor_arc_approximation}.

\subsection{Reduction to mean value theorems}
\label{section:reduction}

The reduction step is based on the following lemma, a special case of Lemma 7 in~\cite{ACHK:WaringGoldbach}. In the present form, the result is a slight variation of Lemma 4.2 in \cite{Hughes_Vinogradov} and is implicit also in \cite{MSW}. 

\begin{lemma}\label{lemma:L2bound}
For $\lambda \in \mathbb N$, let $T_\lambda$ be a convolution operator on $\ell^2(\mathbb Z^d)$ with Fourier multiplier given by
\[ \widehat{T_\lambda}(\boldsymbol\xi) := \int_{\mathfrak{B}} K(\theta; \boldsymbol\xi) e(-\lambda \theta) \, d\theta, \]
where $\mathfrak{B} \subseteq \mathbb T$ is a measurable set and $K( \cdot; \boldsymbol\xi) \in L^1(\mathbb T)$ is a kernel independent of $\lambda$. Further, for $\Lambda \ge 2$, define the dyadic maximal functions  
\[ T_*f(\mathbf x) = T_*(\mathbf x; \Lambda) := \sup_{\lambda \in [\Lambda/2,\Lambda)} |T_\lambda f(\mathbf x)|. \]
Then 
\begin{equation}\label{eq:lem3.1} 
\| T_* \|_{\ell^2(\mathbb Z^d) \to \ell^2(\mathbb Z^d)} \le \int_{\mathfrak{B}} \sup_{\boldsymbol\xi \in \mathbb T^d}  | K(\theta; \boldsymbol\xi) | \, d\theta. 
\end{equation}
\end{lemma}

For a measurable set $\mathfrak B \subset \mathbb T$, we have
\[ 
\widehat{A_\lambda^{\mathfrak B}}(\boldsymbol\xi) 
= \lambda^{1-d/k} \int_{\mathfrak B} \mathcal F_N(\theta; \boldsymbol\xi) e(-\lambda\theta) \, d\theta, \]
where
\[ 
\mathcal F_N(\theta; \boldsymbol\xi) := \prod_{j=1}^{d} S_N(\theta, \xi_j) 
\quad \text{with} \quad S_N(\theta, \xi) := \sum_{|n| \leq N} e(\theta |n|^k + \xi n). \]
Thus, in the proof of Proposition \ref{lemma:minor_arc_approximation}, we apply \eqref{eq:lem3.1} with $K = \mathcal F_N$ and $\mathfrak B = {\mathfrak m}$. The supremum over $\boldsymbol\xi$ on the right side of \eqref{eq:lem3.1} then stands in the way of a direct application of known results from analytic number theory. Our next lemma overcomes this obstacle; its proof is a variant of the argument leading to (12) in Wooley \cite{Wooley_asymptotic}.

\begin{lemma}\label{lemma:meanvaluereduction}
If $k \geq 2$, $l \leq k-1$ and $s$ are natural numbers and $\mathfrak B \subseteq \mathbb T$ a measurable set, then
\begin{equation}\label{eq:lem6.0}
\int_{\mathfrak B} \sup_{\xi \in \mathbb T} |S_N(\theta, \xi) |^{2s} \, d\theta \lesssim N^{l(l+1)/2} \int_{\mathfrak B} \int_{\mathbb T^l} \left| \sum_{n=1}^N e(\theta n^k + \xi_{l} n^{l} + \dots + \xi_1 n) \right|^{2s} d\boldsymbol\xi d\theta + 1.
\end{equation}
\end{lemma}

\begin{proof}
Define
\[ \widetilde{S}_N(\theta, \xi) = \sum_{1 \le n \le N} e(\theta n^k + \xi n). \]
Since
\begin{equation}\label{eq:lem3.2a}
\sup_{\xi \in \mathbb T} |S_N(\theta, \xi)| \le 2\sup_{\xi \in \mathbb T} |\widetilde S_N(\theta, \xi)| + 1, 
\end{equation}
it suffices to establish \eqref{eq:lem6.0} with $\widetilde{S}_N$ in place of $S_N$. 

Set $H_j = sN^j$. For $\mathbf h = (h_1, \dots, h_l) \in \mathbb Z^l$, we define 
\[ a_{\mathbf h}(\theta) = \sum_{1 \le n_1, \dots, n_s \le N} \delta(\mathbf n; \mathbf h) e(\theta \mathfrak f_{s,k}(\mathbf n)), \]
where 
\[ \delta(\mathbf n; \mathbf h) = \begin{cases} 1 & \text{if } \mathfrak f_{s,j}(\mathbf n) = h_j \text{ for all } j = 1, \dots, l, \\ 0 & \text{otherwise}. \end{cases} \]
We have
\[ \widetilde S_N(\theta, \xi)^s = \sum_{h_1 \le H_1} \cdots \sum_{h_l \le H_l} a_{\mathbf h}(\theta)e(\xi h_1), \]
so by applying the Cauchy--Schwarz inequality we deduce that
\[ \sup_{\xi} |\widetilde S_N(\theta, \xi)|^{2s} \le H_1 \cdots H_l \sum_{h_1 \le H_1} \cdots \sum_{h_l \le H_l} |a_{\mathbf h}(\theta)|^2. \]
Hence, 
\begin{equation}\label{eq:lem6.1}
\int_{\mathfrak B} \sup_{\xi} |\widetilde S_N(\theta, \xi) |^{2s} \, d\theta \lesssim N^{l(l+1)/2} \int_{\mathfrak B} \sum_{h_1 \le H_1} \cdots \sum_{h_l \le H_l} a_{\mathbf h}(\theta)\overline{a_{\mathbf h}(\theta)} \, d\theta.
\end{equation} 

We have 
\[ \sum_{h_1 \le H_1} \cdots \sum_{h_l \le H_l} a_{\mathbf h}(\theta)\overline{a_{\mathbf h}(\theta)} = \sum_{1 \le \mathbf n, \mathbf m \le N} e\big(\theta( \mathfrak f_{s,k}(\mathbf n) - \mathfrak f_{s,k}(\mathbf m))\big) \sum_{h_1 \le H_1} \cdots \sum_{h_l \le H_l} \delta(\mathbf n; \mathbf h)\delta(\mathbf m; \mathbf h). \]
By orthogonality,   
\begin{align*} 
\sum_{h_1 \le H_1} \cdots \sum_{h_l \le H_l}  \delta(\mathbf n; \mathbf h)\delta(\mathbf m; \mathbf h) &= \sum_{h_1 \le H_1} \cdots \sum_{h_l \le H_l} \delta(\mathbf m; \mathbf h) \int_{\mathbb T^l} e\bigg( \sum_{j=1}^l \xi_j( \mathfrak f_{s,j}(\mathbf n) - h_j ) \bigg) \, d\boldsymbol\xi \\
&= \int_{\mathbb T^l} e\bigg( \sum_{j=1}^l \xi_j( \mathfrak f_{s,j}(\mathbf n) - \mathfrak f_{s,j}(\mathbf m)) \bigg) \, d\boldsymbol\xi,
\end{align*}
since for a fixed $\mathbf m$, the sum over $\mathbf h$ has exactly one term (in which $h_j = \mathfrak f_{s,j}(\mathbf m)$). Hence,  
\begin{equation}\label{eq:lem6.1a} 
\sum_{h_1 \le H_1} \cdots \sum_{h_l \le H_l}  a_{\mathbf h}(\theta)\overline{a_{\mathbf h}(\theta)} = \int_{\mathbb T^l} \left| \sum_{n=1}^N e( \theta n^k + \xi_ln^l + \dots + \xi_1n) \right|^{2s} d\boldsymbol\xi.
\end{equation} 
The lemma follows from \eqref{eq:lem3.2a}--\eqref{eq:lem6.1a}.
\end{proof}

\begin{remark}\label{rem:3.1}
With slight modifications, the argument of Lemma \ref{lemma:meanvaluereduction} yields also the following estimate
\[ \int_{\mathfrak B} \int_{\mathbb T} |S_N(\theta, \xi) |^{2s} \, d\xi d\theta \lesssim N^{l(l+1)/2-1} \int_{\mathfrak B} \int_{\mathbb T^l} \left| \sum_{n=1}^N e(\theta n^k + \xi_{l} n^{l} + \dots + \xi_1 n) \right|^{2s} d\boldsymbol\xi d\theta + 1. \]
\end{remark}

\subsection{Mean value theorems}
\label{section:VMVT}

We now recall several mean value estimates from the literature on Waring's problem. The first is implicit in the proof of Theorem~10 of Bourgain~\cite{Bourgain_Vinogradov}, which is a variant of a well-known lemma of Hua (see Lemma 2.5 in~\cite{Vaughan}). The present result follows from eqn. (6.6) in \cite{Bourgain_Vinogradov}. We note that when $l = k$, the left side of \eqref{eq:3.6} turns into Vinogradov's integral $J_{s,k}(N)$ and the lemma turns into the main result of Bourgain, Demeter and Guth \cite{BourgainDemeterGuth}.

\begin{lemma}\label{lemma:Bourgain-Hua}
If $k \geq 3$, $2 \leq l \leq k$ and $s \ge \frac 12l(l+1)$ are natural numbers, then 
\begin{equation}\label{eq:3.6}
\int_{\mathbb T} \int_{\mathbb T^{l-1}} \left| \sum_{n=1}^N e(\theta n^k + \xi_{l-1} n^{l-1} + \dots + \xi_1 n) \right|^{2s} d\boldsymbol\xi d\theta \lesssim 
N^{2s - l(l+1)/{2} + \eps}. 
\end{equation}
\end{lemma}

For small $k$, we will use another variant of Hua's lemma due to Br\"udern and Robert~\cite{BrudernRobert}. The following is a weak form of Lemma~5 in~\cite{BrudernRobert}. 

\begin{lemma}\label{lemma:Brudern}
If $k \geq 3$ and $2 \leq l \leq k$ are natural numbers, then
\begin{equation}\label{eq:3.7}
\int_{\mathbb T^2} \left| \sum_{n=1}^N e(\theta n^k + \xi n) \right|^{2^l+2} d\xi d\theta \lesssim N^{2^l-l+1+\eps}. 
\end{equation}
\end{lemma} 

Note that by Remark \ref{rem:3.1} (with $l-1$ in place of $l$) and~Lemma \ref{lemma:Bourgain-Hua} we obtain a version of \eqref{eq:3.7} with $l(l+1)$ in place of $2^l + 2$. Together with Lemma~\ref{lemma:Brudern}, this observation yields the following bound.

\begin{cor}\label{cor:meanT2}
If $k \geq 3$ and $2 \leq l \leq k$ are natural numbers and $r \ge \min \{ l(l+1),2^l+2 \}$, then
\begin{equation}
\int_{\mathbb T^2} \left| \sum_{n=1}^N e(\theta n^k + \xi n) \right|^r d\xi d\theta \lesssim N^{r-l-1+\eps}. 
\end{equation}
\end{cor}

We also use a variant of Lemma \ref{lemma:Bourgain-Hua} that provides extra savings when the integration over $\theta$ is restricted to a set of minor arcs. Lemma \ref{lemma:Wooleyminorarc}, a slight modification of Theorem~1.3 in Wooley~\cite{Wooley_asymptotic}, improves on \eqref{eq:3.6} in the case $l = k-1$. 
Here, $\mathfrak m$ is the set of minor arcs defined at the beginning of~\S\ref{section:previous}.

\begin{lemma}\label{lemma:Wooleyminorarc}
If $k \geq 3$ and $s \ge \frac 12k(k+1)$ are natural numbers, then
\begin{equation}\label{eq:Wooleyminorarc}
\int_{{\mathfrak m}} \int_{\mathbb T^{k-2}} \left| \sum_{n=1}^N e(\theta n^k + \xi_{k-2} n^{k-2} + \dots + \xi_1 n) \right|^{2s} d{\boldsymbol\xi} d\theta \lesssim N^{2s - k(k - 1)/2-2+\eps}. 
\end{equation}
\end{lemma}

\begin{proof}
The main point in the proof of Theorem~2.1 in Wooley~\cite{Wooley_asymptotic} is the inequality (see p.~1495 in~\cite{Wooley_asymptotic})
\[ \int_{{\mathfrak m}} \int_{\mathbb T^{k-2}} \left| \sum_{n=1}^N e(\theta n^k + \xi_{k-2} n^{k-2} + \dots + \xi_1 n) \right|^{2s} d{\boldsymbol\xi} d\theta \lesssim N^{k-2}(\log N)^{2s+1} J_{s,k}(2N). \]
The lemma follows from this inequality and the Bourgain--Demeter--Guth bound for Vinogradov's integral $J_{s,k}(2N)$ (the case $l = k$ of \eqref{eq:3.6}). 
\end{proof}

Now we interpolate between the above bounds. 

\begin{lemma}\label{proposition:interpolation_savings}
If $k \geq 3$ and $2 \le l \le k-1$ are natural numbers and $r$ is real, with 
\begin{equation}\label{eq:p-range}
r_1(k,l) := k^2-\frac{k l - \min(l^2+l, 2^l+2)} {k-l+1} \le r \le k^2+k,
\end{equation}
then  
\begin{equation}\label{eq:savings}
I_{r,k}(N) := \int_{{\mathfrak m}} \sup_{\xi \in \mathbb T} \left| S_N(\theta, \xi) \right|^r d\theta \lesssim N^{r-k-\delta(r) +\eps},
\end{equation}
where $\delta(r)$ is the linear function of $r$ with values $\delta(r_1) = 0$ and $\delta(k^2+k) = 1$.
\end{lemma}

\begin{proof}
Let $r_0 = \min \{ l^2+l, 2^l+2 \}$. The hypothesis on $r$ implies that $r_0 < r \le k^2+k$, so we can find $t \in [0,1)$ such that $r = t r_0 + (1-t) (k^2 + k)$. By Lemma~\ref{lemma:meanvaluereduction} with $l = 1$ and Corollary \ref{cor:meanT2}, 
\begin{equation}\label{eq:I_p0}
I_{r_0,k}(N) \lesssim N \int_{\mathbb T^2} \left| \sum_{n=1}^N e(\theta n^k + \xi n) \right|^{r_0} d\xi d\theta \lesssim N^{r_0 - l + \eps}. 
\end{equation}
On the other hand, Lemma \ref{lemma:meanvaluereduction} with $l = k-2$ and Lemma \ref{lemma:Wooleyminorarc} yield
\begin{equation}\label{eq:I_critical}
I_{k(k+1),k}(N) \lesssim N^{(k-1)(k-2)/2} N^{k(k+1)-k(k-1)/2-2+\eps} = N^{k^2-1+\eps}. 
\end{equation}
Using H\"older's inequality and the above bounds, we get
\begin{align*}
I_{r,k}(N) &\lesssim I_{k(k+1),k}(N)^{1-t} I_{r_0,k}(N)^t \lesssim N^{(1-t)(k^2-1+\eps)} N^{t(r_0-l+\eps)} \\
&\lesssim N^{(1-t)(k^2+k) + t r_0} N^{-(1-t)(k+1)-tl + \eps} = N^{r-k-1+t(k-l+1) + \eps}.
\end{align*}
This inequality takes the form \eqref{eq:savings} with $ \delta(r) = 1 - t(k-l+1)$. Since $t$ depends linearly on $r$ and $t = 0$ when $r = k^2+k$, $\delta(r)$ is a linear function of $r$ with $\delta(k^2+k) = 1$. The value of $r_1(k,l)$ in \eqref{eq:p-range} is the unique solution of the linear equation $\delta(r) = 0$.
\end{proof}

\subsection{Proof of Proposition~\ref{lemma:minor_arc_approximation}}

By Lemma~\ref{lemma:L2bound} and the arithmetic-geometric mean inequality,  
\begin{equation}\label{eq:3.11}
\left\| \sup_{\lambda \in [\Lambda/2,\Lambda)} |A_\lambda^{{\mathfrak m}}| \right\|_{\ell^2(\mathbb Z^d) \to \ell^2(\mathbb Z^d)} \leq N^{k-d} \int_{{\mathfrak m}} \sup_{\boldsymbol\xi \in \mathbb T^{d}} \left| \mathcal F_N(\theta; \boldsymbol\xi) \right| \, d\theta \leq N^{k - d}I_{d,k}(N),
\end{equation}
where $I_{d,k}(N)$ is the integral defined in \eqref{eq:savings}. Thus, the proposition will follow, if we prove the inequality
\begin{equation}\label{eq:needed}
I_{d,k}(N) \lesssim N^{d - k - \delta + \eps}
\end{equation}
with $\delta = k\delta_0(d,k)$.

Let $l_0(k)$ denote the value of $l$ for which the maximum in the definition of $d_0(k)$ is attained (recall \eqref{eq:def_d0}). When $d_0(k) < d \le k^2+k$, we may apply \eqref{eq:savings} with $r=d$ and $l = l_0$ to deduce \eqref{eq:needed} with $\delta = \delta(d)$. We now observe that when $d \le k^2+k$, we have $\delta(d) = k\delta_0(d,k)$ and that the hypothesis $d > d_0(k)$ ensures that $\delta(d) > 0$. 

When $d > k^2+k$, we enhance our estimates with the help of the $L^\infty$-bound for $S_N(\theta, \xi)$ on the minor arcs: by combining a classical result of Weyl (see Lemma 2.4 in Vaughan \cite{Vaughan}) and Theorem~5 in Bourgain \cite{Bourgain_Vinogradov}, we have
\[ \sup_{(\theta, \xi) \in {\mathfrak m} \times \mathbb T} \left| S_N (\theta, \xi) \right| \lesssim
N^{1-\tau_k+\eps},  \]
$\tau_k$ being the quantity that appears in the definition of $\delta_0(d,k)$. Thus, when $d > k^2+k$, we have
\[ I_{d,k}(N) \lesssim N^{(d-k^2-k)(1-\tau_k) + \eps}I_{k(k+1),k}(N) \lesssim N^{d-k-1 - (d-k^2-k) \tau_k + \eps}. \]
We conclude that \eqref{eq:needed} holds with $\delta = 1 + (d - k^2-k)\tau_k$. \qed

\section{Major Arc Analysis}
\label{section:majorarcs}

We will proceed through a series of successive approximations to $A_\lambda^{a/q}$, which we will define by their Fourier multipliers. Our approximations are based on the following major arc approximation for exponential sums that appears as part of Theorem 3 of Br\"udern and Robert \cite{BrudernRobert}. In this result and throughout the section, we write
\[ G(q; a, b) := q^{-1}\sum_{x \in \mathbb Z_q} e_q( {ax^k + bx} ) 
\quad \text{ and } \quad 
v_N(\theta, \xi) := \int_0^N e( \theta t^k + \xi t) \, dt. \]

\begin{lemma}
Let $\theta, \xi \in \mathbb T$, $q \in \mathbb N$, $a \in \mathbb Z_q^*$, and $b \in \mathbb Z$, and suppose that
\[ \left| \theta - \frac aq \right| \leq \frac 1{4kqN^{k-1}}, \quad \left| \xi - \frac bq \right| \leq \frac 1{2q}. \]
Then
\[ \sum_{n \le N} e(\theta n^k + \xi n) = G(q; a,b)v_N(\theta - a/q, \xi - b/q) + O\big( q^{1-1/k+\eps} \big). \]
\end{lemma}

Recall the definition of $S_N(\theta,\xi)$ from Section \ref{section:reduction}. When $a/q + \theta$ lies on a major arc ${\mathfrak M}(a/q)$ and $\xi = b/q + \eta$, with $|\eta| \le 1/(2q)$, the above lemma yields
\begin{equation}\label{eq:majorapprox} 
S_N(a/q+\theta, \xi) = G(q; a,b)v_N(\theta, \eta) + G(q; a,-b)v_N(\theta, -\eta) + O\big( q^{1-1/k+\eps} \big).
\end{equation}
We will use this approximation in conjunction with the following well-known bounds (see Theorems 7.1 and 7.3 in Vaughan \cite{Vaughan}):
\begin{equation}\label{eq:basicbounds}
G(q; a, b) \lesssim q^{-1/k+\eps}, \quad v_N(\theta, \eta) \lesssim N(1 + N|\eta| + N^k|\theta|)^{-1/k}.
\end{equation}
We will make use also of the inequality
\begin{equation}\label{eq:basicbound2}
v_N(\theta, \eta) \lesssim N(1 + N|\eta|)^{-1/2}, 
\end{equation}
which follows from the representation
\[ v_N(\theta, \eta) = \frac 1k \int_0^{N^k} u^{1/k-1} e(\theta u + \eta u^{1/k}) \, du \]
and the second-derivative bound for oscillatory integrals (see p. 334 in \cite{SteinHA}).

\subsection*{Proof of Proposition \ref{lemma:majorarcs}}

The bulk of the work concerns the case $p=2$ of the proposition. 
Since
\[ \widehat{A_\lambda^{a/q}}(\boldsymbol\xi) =
\lambda^{1-d/k} e_q(-\lambda a) \int_{{\mathfrak M}(0/q)} \mathcal F_N(a/q + \theta; \boldsymbol\xi) e(-\lambda \theta) \, d\theta, \]
the asymptotic \eqref{eq:majorapprox} suggests that the following multiplier should provide a good approximation to $\widehat{A_{\lambda}^{a/q}}(\boldsymbol\xi)$:
\[ \widehat{B_{\lambda}^{a/q}}(\boldsymbol\xi) := \lambda^{1-d/k} e_q(-\lambda a) \int_{{\mathfrak M}(0/q)} \mathcal{G}_N(\theta; \boldsymbol\xi)e(-\lambda\theta) \, d\theta, \]
where 
\[ \mathcal{G}_N(\theta; \boldsymbol\xi) := \prod_{j=1}^d \big\{ G(q; a, b_j)v_N(\theta, \eta_j) + G(q; a, -b_j)v_N(\theta, -\eta_j)\big\}, \]
with $b_j$ the unique integer such that $-\frac 12 \le b_j - q\xi_j < \frac 12$ and $\eta_j = \xi_j - b_j/q$. Let $B_\lambda^{a/q}$ denote the operator on $\ell^2(\mathbb Z^d)$ with the above Fourier multiplier. To estimate the $\ell^2$-error of approximation of $A_{\lambda}^{a/q}$ by $B_{\lambda}^{a/q}$, we use that when $\theta \in {\mathfrak M}(0/q)$, \eqref{eq:majorapprox} and \eqref{eq:basicbounds} yield
\[ | \mathcal F_N(a/q+\theta; \boldsymbol\xi) - \mathcal{G}_N(\theta; \boldsymbol\xi) | \lesssim q^{1-d/k+\eps}N^{d-1}(1 + N^k|\theta|)^{(1-d)/k}. \]
Thus, if $d > k+1$, we have
\begin{align*} 
\int_{{\mathfrak M}(0/q)} \sup_{\boldsymbol\xi \in \mathbb T^d} | \mathcal F_N(a/q+\theta; \boldsymbol\xi) - \mathcal{G}_N(\theta; \boldsymbol\xi) | \, d\theta &\lesssim \int_{\mathbb R} \frac {q^{1-d/k+\eps}N^{d-1} \, d\theta}{(1 + N^k|\theta|)^{(d-1)/k}} \\
&\lesssim q^{1-d/k+\eps}N^{d-k-1}.
\end{align*}
Lemma \ref{lemma:L2bound} then yields
\begin{equation}\label{eq:4.3}
\left\| \sup_{\lambda \in [\Lambda/2,\Lambda)} \big| A_\lambda^{a/q} - B_\lambda^{a/q} \big| \right\|_{\ell^2(\mathbb Z^d) \to \ell^2(\mathbb Z^d)} \lesssim q^{1-d/k+\eps}N^{-1}.
\end{equation}

Next, we approximate $B_{\lambda}^{a/q}$ by the operator $C_{\lambda}^{a/q}$ with multiplier
\[ \widehat{C_{\lambda}^{a/q}}(\boldsymbol\xi) := \lambda^{1-d/k} e_q(-\lambda a) \sum_{\mathbf b \in \mathbb Z^d} \psi(q\boldsymbol\xi - \mathbf b)  \int_{{\mathfrak M}(0/q)} \mathcal{G}_N(\theta; \boldsymbol\xi)e(-\lambda\theta) \, d\theta. \]
By the localization of $\psi$, the above sum has at most one term in which $\mathbf b$ matches the integer vector that appears in the definition of $\mathcal G_N(\theta; \boldsymbol\xi)$. Hence, 
\[ \mathcal G_N(\theta; \boldsymbol\xi)\bigg( 1 - \sum_{\mathbf b \in \mathbb Z^d} \psi(q\boldsymbol\xi - \mathbf b) \bigg) \]
is supported on a set where $\frac 18 \le |q\xi_j - b_j| \le \frac 12$ for some $j$. For such $j$, by \eqref{eq:basicbound2},
\[ v_N(\theta, \xi_j - b_j/q) \lesssim (qN)^{1/2}, \]
and we conclude that
\[ |\mathcal G_N(\theta; \boldsymbol\xi)|\bigg( 1 - \sum_{\mathbf b \in \mathbb Z^d} \psi(q\boldsymbol\xi - \mathbf b) \bigg) \lesssim q^{1/2-d/k+\eps}N^{d-1/2}(1 + N^k|\theta|)^{(1-d)/k}. \]
So, when $d > k+1$, Lemma \ref{lemma:L2bound} gives
\begin{align}\label{eq:4.4}
\left\| \sup_{\lambda \in [\Lambda/2,\Lambda)} \big| B_\lambda^{a/q} - C_\lambda^{a/q} \big| \right\|_{\ell^2(\mathbb Z^d) \to \ell^2(\mathbb Z^d)} &\lesssim \int_{\mathbb R} \frac {q^{1/2-d/k+\eps}N^{k-1/2} \, d\theta}{(1 + N^k|\theta|)^{(d-1)/k}} \notag\\
&\lesssim q^{1/2-d/k+\eps}N^{-1/2}.
\end{align}

In our final approximation, we replace $C_{\lambda}^{a/q}$ by the operator $D_{\lambda}^{a/q}$ with multiplier
\[ \widehat{D_{\lambda}^{a/q}}(\boldsymbol\xi) := \lambda^{1-d/k} e_q(-\lambda a) \sum_{\boldsymbol w \in \{ \pm 1 \}^d} \sum_{\mathbf b \in \mathbb Z^d} \psi(q\boldsymbol\xi - \mathbf b) G(q; a, \boldsymbol w\mathbf b) J_{\lambda}(\boldsymbol w(\boldsymbol\xi - q^{-1}\mathbf b)), \]
where $\boldsymbol w\mathbf b = (w_1b_1, \dots, w_db_d)$ and
\[ J_{\lambda}(\boldsymbol\eta) := \int_{\mathbb R} \bigg\{ \prod_{j=1}^{d} v_N(\theta; \eta_j) \bigg\} e(-\lambda\theta) \, d\theta. \]
We remark that $\widehat{C_{\lambda}^{a/q}}(\boldsymbol\xi)$ can be expressed in a matching form, with $J_{\lambda}(\boldsymbol\xi - q^{-1}\boldsymbol w\mathbf b)$ replaced by the analogous integral over ${\mathfrak M}(0/q)$. 
Thus, when $d > k$, we deduce from Lemma \ref{lemma:L2bound} and \eqref{eq:basicbounds} that
\begin{align}\label{eq:4.5}
\left\| \sup_{\lambda \in [\Lambda/2,\Lambda)} \big| C_\lambda^{a/q} - D_\lambda^{a/q} \big| \right\|_{\ell^2(\mathbb Z^d) \to \ell^2(\mathbb Z^d)} &\lesssim \int_{{\mathfrak M}(0/q)^c} \frac {q^{-d/k+\eps}N^{k} \, d\theta}{(1 + N^k|\theta|)^{d/k}} \notag\\
&\lesssim q^{-1+\eps}N^{1-d/k}.
\end{align}
Here, ${\mathfrak M}(0/q)^c$ denotes the complement of the interval ${\mathfrak M}(0/q)$ in $\mathbb R$. 

Finally, we note that $D_\lambda^{a/q}$ is really $M_\lambda^{a/q}$.  Indeed, by the discussion on p. 498 in Stein \cite{SteinHA} (see also Lemma 5 in  Magyar \cite{Magyar:ergodic}), we have 
\begin{align}\label{sphericalmeasure}
J_\lambda(\boldsymbol\eta) &= \int_{\mathbb R} \int_{\mathbb R^d} \mathbf 1_{[0,N]^d}(\mathbf t)e(\boldsymbol\eta \cdot \mathbf t) e(\theta(\mathfrak{f}(\mathbf t) - \lambda)) \, d\mathbf t d\theta \notag\\
&= \lambda^{d/k-1} \int_{\mathbb R^d} \mathbf 1_{[0,N]^d}(\mathbf t)e(\boldsymbol\eta \cdot \mathbf t) \, d\sigma_\lambda(\mathbf t) = \lambda^{d/k-1} \widetilde{d\sigma_\lambda}(\boldsymbol\eta), 
\end{align}
since the surface measure $d\sigma_\lambda$ is supported in the cube $[0,N]^d$. Combining this observation and \eqref{eq:4.3}--\eqref{eq:4.5}, and summing over $a,q$, we conclude that when $d > 2k$, 
\begin{equation}\label{eq:4.8}
\sum_{q \le N} \sum_{a \in \mathbb Z_q^*} \left\| \sup_{\lambda \in [\Lambda/2,\Lambda)} \big| A_\lambda^{a/q} - M_\lambda^{a/q} \big| \right\|_{\ell^2(\mathbb Z^d) \to \ell^2(\mathbb Z^d)} \lesssim N^{-\gamma+\eps}, 
\end{equation} 
with $\gamma = \min \left(d/k-2, \frac 12 \right) > 0$. In particular, when $d \ge \frac 52 k$, we have $\gamma = \frac 12$. 

We can now finish the proof of Proposition \ref{lemma:majorarcs}. For brevity, we write $p_1 = \frac{d}{d-k}$. From \eqref{eq:4.8}, we obtain that
\begin{equation}\label{interpolate_major_arcs1}
\left\| \sup_{\lambda \in [\Lambda/2,\Lambda)} \big| A_\lambda^{{\mathfrak M}} - M_\lambda \big| \right\|_{\ell^2(\mathbb Z^d) \to \ell^2(\mathbb Z^d)} \lesssim N^{-\gamma+\eps}. 
\end{equation}
On the other hand, we know from Propositions \ref{lemma:major_arc_bound} and \ref{lemma:mainterm} that both $ A_{\Lambda}^{{\mathfrak M}}$ and $M_*$ are bounded on $\ell^p(\mathbb Z^d)$ when $p_1 < p \le 2$; hence,
\begin{equation}\label{interpolate_major_arcs2}
\left\| \sup_{\lambda \in [\Lambda/2, \Lambda)} \big| A_\lambda^{{\mathfrak M}}-M_\lambda \big| \right\|_{\ell^p(\mathbb Z^d) \to \ell^p(\mathbb Z^d)} \lesssim 1.
\end{equation}
When $p \in (p_0(d,k), 2]$, we interpolate between \eqref{interpolate_major_arcs1} and \eqref{interpolate_major_arcs2} with $p = p_1 + \eps$ for a sufficiently small $\eps > 0$. The interpolation yields \eqref{eq:majorarcs} with $\beta_p > 0$ that can be chosen arbitrarily close to $2\gamma(p-p_1)/ (kp(2-p_1))$. \qed

\section{Main Term Contribution}
\label{sec:5}

In this section, we prove Proposition \ref{lemma:mainterm}. First, we obtain $L^p(\mathbb R^d)$-bounds for the maximal function of the continuous surface measures $d\sigma_\lambda$.

\begin{lemma}\label{lemma:Archimedean_max_fxn}
If $k \ge 2$, $d>\frac32k$ and $p > \frac{2d-k}{2d - 2k}$, then for all $f \in L^p(\mathbb R^d)$, 
\begin{equation}\label{eq:Rubio} 
\left\| \sup_{\lambda > 0} |f * d\sigma_\lambda| \right\|_{L^p(\mathbb R^d)} \lesssim \|f\|_{L^p(\mathbb R^d)}. 
\end{equation}
\end{lemma}

\begin{proof}
We will deduce the lemma from a result of Rubio de Francia -- Theorem A in \cite{Rubio} -- which reduces \eqref{eq:Rubio} to bounds for the Fourier transform of the measure $d\sigma_\lambda$. First, we majorize the measure $d\sigma_\lambda$ by a smooth one. By the choice of normalization of $d\sigma_\lambda$, we have
\[ f * d\sigma_\lambda(\mathbf x) = \int_{\mathbb R^d} f(\mathbf x - \mathbf y) \, d\sigma_\lambda(\mathbf y) = \int_{\mathbb R^d} f(\mathbf x - t\mathbf y) \, d\sigma_1(\mathbf y), \]
where $t = \lambda^{1/k}$. Let $\phi$ be a smooth function supported in $\left[-\frac 12, \frac 32\right]$ and such that $\mathbf 1_{[0,1]}(x) \le \phi(x)$, and write $\phi(\mathbf x) = \phi(x_1) \cdots \phi(x_d)$. Since $d\sigma_1$ is supported inside the unit cube $[0,1]^d$, we have
\[ \bigg|\int_{\mathbb R^d} f(\mathbf x - t\mathbf y) \, d\sigma_1(\mathbf y) \bigg| \le \int_{\mathbb R^d} |f(\mathbf x - t\mathbf y)| \phi(\mathbf y) \, d\sigma(\mathbf y) =: \int_{\mathbb R^d} |f(\mathbf x - t\mathbf y)| \, d\mu(\mathbf y), \]
where $d\sigma$ is the surface measure on the smooth manifold $x_1^k + \dots + x_d^k = 1$. By Rubio de Francia's theorem, the maximal function
\[ \mathcal A_tf(\mathbf x) := \sup_{t > 0} \int_{\mathbb R^d} |f(\mathbf x - t\mathbf y)| \, d\mu(\mathbf y) \]
is bounded on $L^p(\mathbb R^d)$, provided that
\begin{equation}\label{eq:mu_tilde}
\widetilde{\mu}(\boldsymbol\xi) \lesssim (|\boldsymbol\xi| + 1)^{-a} \quad \text{for some } a > \min\left\{\frac 1{2p-2},\frac12\right\}.
\end{equation} 
Thus, the lemma will follow, if we establish \eqref{eq:mu_tilde} with $a = d/k - 1$ and $d > k+1$ because $a>1/2$ once $d>\frac32k$.

We now turn to \eqref{eq:mu_tilde}. Similarly to \eqref{sphericalmeasure}, we have
\[ \widetilde{\mu}(\boldsymbol\xi) = \int_{\mathbb R} \int_{\mathbb R^d} \phi(\mathbf x)e(\boldsymbol\xi \cdot \mathbf x) e(\theta(\mathfrak f(\mathbf x)-1)) \, d\mathbf x d\theta = \int_{\mathbb R} \bigg\{ \prod_{j=1}^d v_\phi(\theta, \xi_j) \bigg\}e(-\theta) \, d\theta, \]
where
\[ v_\phi(\theta, \xi) := \int_{\mathbb R} \phi(x)e(\theta x^k + \xi x) \, dx. \]
By the corollary on p. 334 of Stein \cite{SteinHA}, we have
\begin{equation}\label{eq:vphi1} 
v_\phi(\theta, \xi) \lesssim (1 + |\theta|)^{-1/k},
\end{equation} 
uniformly in $\xi$. On the other hand, if $k2^k|\theta| \le |\xi|$, we have
\[ \bigg| \frac d{dx} \big( \xi^{-1}\theta x^k + x \big) \bigg| \ge \frac 12, \quad \bigg| \frac {d^j}{dx^j} \big( \xi^{-1}\theta x^k + x \big) \bigg| \lesssim_j 1 \]
on the support of $\phi$. Hence, Proposition VIII.1 on p. 331 of Stein \cite{SteinHA} yields
\begin{equation}\label{eq:vphi2} 
v_\phi(\theta, \xi) \lesssim_M (1 + |\xi|)^{-M}
\end{equation} 
for any fixed $M \ge 1$. 

We now choose an index $j$, $1 \le i \le d$, such that $|\boldsymbol\xi| \le d|\xi_i|$ and set $\theta_0 = |\xi_i|/(k2^k)$. We apply \eqref{eq:vphi1} to the trigonometric integrals $v_\phi(\theta,\xi_j)$, $j \ne i$, and to $v_\phi(\theta,\xi_i)$ when $|\theta| > \theta_0$; we apply \eqref{eq:vphi2} to $v_\phi(\theta,\xi_i)$ when $|\theta| \le \theta_0$. From these bounds and the integral representation for $\widetilde{\mu}(\boldsymbol\xi)$, we obtain
\begin{align*} 
\widetilde{\mu}(\boldsymbol\xi) &\lesssim \int_{|\theta| \le \theta_0} \frac {(1 + |\xi_i|)^{-M}}{(1 + |\theta|)^{(d-1)/k}} \, d\theta + \int_{|\theta| > \theta_0} \frac {d\theta}{(1 + |\theta|)^{d/k}} \\
&\lesssim (1 + |\xi_i|)^{-M} + (1 + |\xi_i|)^{1-d/k} \lesssim (1+|\boldsymbol\xi|)^{1-d/k}, 
\end{align*}
provided that $M \ge d/k - 1$ and $d > k+1$. 
This completes the proof.
\end{proof}

\subsection*{Proof of Proposition~\ref{lemma:mainterm}}

To prove Proposition~\ref{lemma:mainterm}, it suffices to prove that (uniformly in $a$ and $q$)
\[ 
\left\| \sup_{\lambda \in \mathbb N} \big| M_\lambda^{a/q} \big| \right\|_{\ell^p(\mathbb Z^d) \to \ell^p(\mathbb Z^d)} 
\lesssim 
q^{-\frac{d}{k}(2-2/p)+\eps},
\]
for all $\frac{d}{d-k} < p \leq 2$ and $d>k$. The proposition then follows by summing over $a$ and $q$ (the hypothesis on $p$ ensures that the resulting series over $q$ is convergent). 

Fix $q \in \mathbb N$ and $a \in \mathbb Z_q^*$ and write $\psi_1(\boldsymbol\xi) = \psi(\boldsymbol\xi/2)$ (so that $\psi = \psi \psi_1$). We borrow a trick from Magyar, Stein and Wainger \cite{MSW} to express $\widehat{M_\lambda^{a/q}} (\boldsymbol\xi)$ as a linear combination of Fourier multipliers that separate the dependence on $\lambda$ and from the dependence on $a/q$: 
\begin{align*}
\widehat{M_\lambda^{a/q}}(\boldsymbol\xi) &= e_q(-\lambda a) \sum_{\boldsymbol w \in \{ \pm 1 \}^d} \sum_{\mathbf b \in \mathbb Z^d} \psi(q\boldsymbol\xi - \mathbf b) G(q; a, \boldsymbol w\mathbf b) \widetilde{d\sigma_{\lambda}}(\boldsymbol w(\boldsymbol\xi - q^{-1}\mathbf b)) \\ 
&= e_q(-\lambda a) \sum_{\boldsymbol w \in \{ \pm 1 \}^d} \bigg( \sum_{\mathbf b \in \mathbb Z^d} \psi_1(q\boldsymbol\xi - \mathbf b) G(q; a, \boldsymbol w\mathbf b) \bigg) \bigg( \sum_{\mathbf b \in \mathbb Z^d} \psi(q\boldsymbol\xi - \mathbf b) \widetilde{d\sigma_{\lambda}}(\boldsymbol w(\boldsymbol\xi - q^{-1}\mathbf b)) \bigg) \\ 
& =: e_q(-\lambda a) \sum_{\boldsymbol w \in \{ \pm 1 \}^d} \widehat{S_{\boldsymbol w}^{a/q}}(\boldsymbol\xi) \widehat{T_{\lambda,\boldsymbol w}^{q}}(\boldsymbol\xi).
\end{align*}
Since $\boldsymbol w$ takes on precisely $2^d$ values for each $a/q$, it suffices to prove that 
\begin{equation}\label{eq:mainterm_majorarc_estimate:quadrant}
\left\| \sup_{\lambda \in \mathbb N} \big|T_{\lambda,\boldsymbol w}^{q} \circ S_{\boldsymbol w}^{a/q} \big| \right\|_{\ell^p(\mathbb Z^d) \to \ell^p(\mathbb Z^d)} 
\lesssim q^{-\frac{d}{k}(2-2/p)+\eps},
\end{equation}
uniformly for $\boldsymbol w \in \{-1,1\}^d$. To prove \eqref{eq:mainterm_majorarc_estimate:quadrant}, we will first bound the maximal function over the `Archimedean' multipliers $T_{\lambda,\boldsymbol w}^{q}$, and then we will bound the non-Archimedean multipliers $S_{\boldsymbol w}^{a/q}$. 
This is possible because $S_{\boldsymbol w}^{a/q}$ is independent of $\lambda \in \mathbb N$. 

For $d > \frac32k$ and $p>\frac{2d-k}{2d-2k}$, Lemma~\ref{lemma:Archimedean_max_fxn} and Corollary~2.1 of \cite{MSW} (the `Magyar--Stein--Wainger transference principle') give the bound
\begin{equation}\label{eq:transferbound}
\left\| \sup_{\lambda \in \mathbb N} \big| T_{\lambda,\boldsymbol w}^{q}g \big| \right\|_{\ell^p(\mathbb Z^d)} \lesssim \|g\|_{\ell^p(\mathbb Z^d)},
\end{equation}
with an implicit constant independent of $q$ and $\boldsymbol w$. We now observe that $S_{\boldsymbol w}^{a/q}$ does not depend on $\lambda$ and apply \eqref{eq:transferbound} with $g = S_{\boldsymbol w}^{a/q}f$ to find that 
\begin{equation}\label{eq:4.13} 
\left\| \sup_{\lambda \in \mathbb N} \big|T_{\lambda,\boldsymbol w}^{q} \circ S_{\boldsymbol w}^{a/q} \big| \right\|_{\ell^p(\mathbb Z^d) \to \ell^p(\mathbb Z^d)} \lesssim \big\| S_{\boldsymbol w}^{a/q} \big\|_{\ell^p(\mathbb Z^d) \to \ell^p(\mathbb Z^d)},
\end{equation}
under the same assumptions on $d$ and $p$ which we note are weaker than the hypotheses of the proposition. Finally, observe that $G(q; a, b)$ is a $q$-periodic function with $\mathbb Z_q$-Fourier transform equal to 
\[ \sum_{b \in \mathbb Z_q} e_q(-mb) G(q; a, b) = 
q^{-1} \sum_{x \in \mathbb Z_q} e_q( ax^k ) \sum_{b \in \mathbb Z_q} e_q( b(x-m) ) = e_q( {am^k} ), \]
for each $m \in \mathbb Z_q$. Hence, we may apply Proposition~2.2 in \cite{MSW} and the bound \eqref{eq:basicbounds} for $G(q; a, b)$ to deduce that 
\begin{equation}\label{eq:4.14} 
\big\|S_{\boldsymbol w}^{a/q} \big\|_{\ell^p(\mathbb Z^d) \to \ell^p(\mathbb Z^d)} \lesssim q^{-\frac{d}{k}(2-2/p)+\eps}.
\end{equation}
The desired inequality \eqref{eq:mainterm_majorarc_estimate:quadrant} follows immediately from \eqref{eq:4.13} and \eqref{eq:4.14}. \qed

\section*{Acknowledgments}
The first author was supported by NSF grant DMS-1502464. The second author was supported by NSF grant DMS-1147523 and by the Fields Institute. Last but not least, the authors thank Alex Nagel and Trevor Wooley for several helpful discussions.

\bibliographystyle{amsplain}

\begin{dajauthors}
\begin{authorinfo}[TA]
Theresa C. Anderson\\
Department of Mathematics\\ 
University of Wisconsin--Madison\\ 
Madison, WI 53705, U.S.A.\\
tcanderson\imageat{}math\imagedot{}wisc\imagedot{}edu  
\end{authorinfo}
\begin{authorinfo}[BC]
Brian Cook\\
Department of Mathematicl Sciences\\ 
Kent State University\\ 
Kent, OH 44242\\
U.S.A.\\
bcook25\imageat{}kent\imagedot{}edu 
\end{authorinfo}
\begin{authorinfo}[KH]
Kevin Hughes\\
School of Mathematics\\	
The University of Bristol\\	
Bristol, BS8 1TW, UK\\
Kevin.Hughes\imageat{}bristol\imagedot{}ac\imagedot{}uk 
\end{authorinfo}
\begin{authorinfo}[AK]
Angel Kumchev\\
Department of Mathematics\\	
Towson University\\	
Towson, MD 21252, U.S.A.\\
akumchev\imageat{}towson\imagedot{}edu
\end{authorinfo}
\end{dajauthors}

\end{document}